\documentclass[12pt]{amsart}

\usepackage{amsfonts}
\usepackage{color}
\usepackage{amscd}
\usepackage{epsfig}
\usepackage{graphicx}
\usepackage{amsmath,amssymb}
\usepackage{enumerate}
\usepackage{amsthm}
\usepackage{url,hyperref}

\newtheorem{thm}{Theorem}[section]
\newtheorem{lem}[thm]{Lemma}

\newtheorem{conj}[thm]{Conjecture}

\newtheorem{prop}[thm]{Proposition}
\newtheorem{cor}[thm]{Corollary}

\newtheorem{rmk}[thm]{Remark}

\newtheorem{ques}[thm]{Question}

\newtheorem{ex}[thm]{Example}

\def\O{{\mathcal O}}

\def\P{{\mathbb P}}

\def\I{{\mathcal I}}

\def\Z{{\mathbb Z}}

\def\R{\mathbb R}
\def\C{{\mathbb C}}

\def\Pic{\mathop{\rm Pic}}
\def\Sing{\mathop{\rm Sing}}

\def\Cl{\mathop{\rm Cl}}

\def\Proj{\mathop{\rm Proj}}

\def\Spec{\mathop{\rm Spec}}

\def\fm{\mathfrak m}
\def\fn{\mathfrak n}

\def\supp{\mathop{\rm Supp}}

\def\deg{\mathop{\rm deg}}

\def\ss{\vskip .2 in}

\def\blu{\textcolor{blue}}
\def\ss{\vskip.10in}

\def\rdA{\mathbf A}
\def\rdD{\mathbf D}
\def\rdE{\mathbf E}

\title{Geometric divisors in normal local domains}

\author{John Brevik}

\address{California State University at Long Beach, 
Department of Mathematics and Statistics, Long Beach, CA 90840}

\email{john\_brevik@csulb.edu}

\author{Scott Nollet}

\address{Texas Christian University, Department of Mathematics, 
Fort Worth, TX 76129}

\email{s.nollet@tcu.edu}


\begin{document}
\bibliographystyle{plain}

\begin{abstract} 
Let $A$ be the local ring at a point of a normal complex  
variety with completion $R = \widehat A$. 
Srinivas has asked about the possible images of the induced map 
$\Cl A \to \Cl R$ over all geometric normal domains $A$ with fixed completion $R$. 
We use Noether-Lefschetz theory to prove that all finitely generated subgroups 
are possible in some familiar cases.  
As a byproduct we show that every finitely generated abelian group appears as 
the class group of the local ring at the vertex of a cone over some 
smooth complex variety of each positive dimension. 
\end{abstract}

\maketitle

\section{Introduction}

In his survey on geometric methods in commutative 
algebra \cite{srinivas}, Srinivas poses several interesting questions about {\it geometric local rings}, 
the local rings $A = \O_{X,x}$ of a point $x$ on a complex algebraic variety $X$, 
or equivalently local domains of essentially finite type over $\mathbb C$. 
For $A$ normal with completion $R = \widehat A$, there is a natural map  

\begin{equation}\label{inject}
 \iota: \Cl A \hookrightarrow \Cl R
\end{equation}
whose injectivity is attributed to Mori \cite[Thm. 6.5]{samuel}, so we call $\iota$ the {\it Mori map}. 
Many non-isomorphic geometric rings can have isomorphic completions. 
For example, the local rings at smooth points of $d$-dimensional 
complex varieties have isomorphic completions by Cohen's theorem~\cite[Thm. 60, Cor. 2]{M1},
but unless the varieties are birational they do not share the same fraction field. 
Srinivas asks about the images of the inclusion (\ref{inject}) as $A$ varies with 
fixed completion $R$, the most ambitious and open-ended question being: 

\begin{ques}\label{Q1}{\em
Which subgroups of $\Cl R$ arise as images 
$\Cl A \hookrightarrow \Cl R$ where $R \cong \widehat A$?  
\em}\end{ques}

\begin{ex}\label{infinite}{\em
Srinivas notes that not all images are possible, because the complete local ring 
$R = \mathbb C [[x,y,z]]/(x^2+y^3+z^5)$ has class group $\Cl R \cong \mathbb C$, 
but for every geometric local ring $A$ with $\widehat A \cong R$, the image 
$\Cl A \hookrightarrow \Cl R$ must be finitely generated \cite[Ex. 3.9]{srinivas}. 
He reasons that if $A = \O_{X,x}$ for a surface $X$ and $Y \to X$ is a resolution of 
singularities, then the induced map $\Pic Y \to \Cl A$ is surjective. 
Since $\Pic^0 Y$ is projective, it 
has trivial image in the affine group ${\mathbb G}_a = \mathbb C$, therefore 
$\Cl A \to \mathbb C$ factors through the finitely generated Neron-Severi group. 
\em}\end{ex}
 
\begin{rmk}\label{notC}{\em
Example \ref{infinite} shows that in fact $\mathbb C$ is never the class group of 
any geometric local ring. More generally, the argument applies to any infinitely 
generated group whose connected component of the identity is an affine group. 
\em}\end{rmk} 

In view of Example \ref{infinite}, Srinivas concludes that ``probably the only 
reasonable general question" one can ask is for the {\em smallest} possible images. 
If $A$ is the codimension $r$ quotient of a regular local ring $B$, the dualizing 
module $\omega_A = \text{Ext}_A^r (B,A) \in \Cl R$ is trivial if and only if 
$R$ is Gorenstein \cite{murthy}. This class is independent of $B$ and the image in 
$\Cl R$ is independent of $A$, because 
$\widehat{\omega_A} \cong \omega_{\widehat A}$ \cite[Thm. 3.3.5(c)]{BH}.
These considerations led Srinivas to ask the following \cite[Ques. 3.5 and 3.7]{srinivas}: 

\begin{ques}\label{Q2}{\em
Let $R$ be the completion of a normal geometric local ring. 
\begin{enumerate}
\item[(a)] If $R$ is Gorenstein, is $R$ the completion of a UFD?
\item[(b)] More generally, does there exist $B$ for which 
$\Cl B  = \langle \omega_R \rangle \subset \Cl R$?  
\end{enumerate}
\em}\end{ques}

Here there has been progress, starting with Grothendieck's solution 
\cite[XI, Cor. 3.14]{SGA2} to Samuel's conjecture stating that a local 
complete intersection ring that is factorial in codimension $\le 3$ is 
a UFD. 
Hartshorne and Ogus proved that if $R$ has an isolated singularity, depth $\geq 3$ and 
embedding dimension at most $2 \dim R - 3$, then $R$ is a UFD \cite{HO}.
Expanding on Srinivas' observation that local rings of rational double 
point surface singularities are completions of UFDs \cite{srinivas87}, 
Parameswaran and Srinivas \cite{parasrini} used methods along the lines of 
Lefschetz' proof of Noether's theorem to construct a UFD defining a singularity 
analytically isomorphic to isolated complete intersection singularity 
of dimension $2$ or $3$. We gave a positive answer for hypersurface singularities 
of dimension $\geq 2$ \cite{BN4} and arbitrary local complete intersection singularities of dimension $\geq 3$ \cite{BN5}. The only result we've seen in the 
non-Gorenstein case is due to van Straten and Parameswaran \cite{pvs}, 
who showed that Question \ref{Q2} (a) has a positive answer if $\dim R = 2$. 
Heitmann has characterized completions of UFDs \cite{heitmann}, 
but his constructions are rarely geometric. 

While Question \ref{Q2} has received much attention, Question \ref{Q1} remains wide open. 
In an effort to better understand it, 
we call an element $\alpha \in \Cl R$ a 
{\it geometric divisor} if it is in the image of the 
inclusion (\ref{inject}) for some geometric local ring $A$ with 
$R = \widehat A$. In view of Example \ref{infinite}, we pose the following:

\begin{ques}\label{Q3}{\em Let $R$ be the completion of a normal geometric local ring. 
\begin{enumerate}
\item[(a)] Given any finitely generated group $G \subset \Cl R$, 
is there a geometric local ring $B$ with $\widehat B = R$ and $G = \Cl B$? 
\item[(b)] Given $\alpha_1, \dots \alpha_r \in \Cl R$, is there a geometric local 
ring $B$ with $\widehat B = R$ and 
$\alpha_i \in \Cl B$ for each $1 \leq i \leq r$? 
\item[(c)] Is every $\alpha \in \Cl R$ a geometric divisor? 
\item[(d)] Do the geometric divisors form a subgroup of $\Cl R$? 
\end{enumerate}
\em}\end{ques}

Note the implications $(a) \Rightarrow (b) \Rightarrow (c) \Rightarrow (d)$. 
Our ignorance about the nature of geometric divisors is revealed in part (c): 
could there be transcendental divisors 
that cannot be accessed geometrically? 
The methods of \cite{BN4,BN5} inspire the following possibility: 

\begin{conj}\label{whynot}
Let $R$ be the completion of a normal geometric complete intersection ring. 
Then for any finitely generated group $G \subset \Cl R$, there is a geometric local 
ring $B$ with $\widehat B = R$ and $G = \Cl B \subset \Cl R$. 
\end{conj} 

In the best-understood case that $R$ is a completed rational double 
point surface singularity, we used our Noether-Lefschetz theorem \cite{BN} with 
carefully chosen base loci to show that every subgroup of $\Cl R$ is the image of 
$\Cl A$ for the local ring $A = \O_{S,p}$ of a surface 
$S \subset \mathbb P^3$ \cite{BN4}, thereby proving Conjecture \ref{whynot} 
for two dimensional rational double point singularities. 

\begin{rmk} {\em This result contrasts sharply with case in which the function field is rational, 
where Mohan Kumar \cite{MK} shows that for most $\rdA_n$ and $\rdE_n$ singularities there is in fact only one possible isomorphism class for the local ring (and thus the class group). The three exceptions, with two possibilities each, are the $\rdE_8$, $\rdA_7$, and $\rdA_8$; for all other $\rdE_n$ and $\rdA_n$ singularities, the Mori map is an isomorphism.  Since an $\rdE_8$ is a UFD under completion, any $\rdE_8$ is a UFD. By following Mohan Kumar's constructions of the $\rdA_7$ and $\rdA_8$ carefully, one sees that the image of the Mori map for the $\rdA_7$ is either the full completed class group $\Z/8\Z$ or the subgroup of order $4$, while in the $\rdA_8$ case the Mori map is either surjective onto $\Z/9\Z$ or its image is of order $3$.}
\end{rmk}

In this paper we give more positive evidence for Conjecture~\ref{whynot}. 
Our main tool is the following theorem, which proves $(b) \Rightarrow (a)$ in 
Question \ref{Q3} for most complete intersection domains.

\begin{thm}\label{main} Let $p \in V \subset \mathbb P^n$ be a normal complete intersection 
point on a complex variety, where $n=3$ if $\dim V =2$. 
Then for any finitely generated subgroup $G \subset \Cl \O_{V,p}$, 
there is a geometric local ring $B$ with 
$\widehat B = \widehat \O_{V,p}$ and $\Cl B = G \subset \Cl \O_{V,p} \subset \Cl \widehat \O_{V,p}$. 
\end{thm}

\begin{rmk}\label{condition}{\em 
The restriction $n=3$ if $\dim V=2$ shows the reliance of our argument on the Noether-Lefschetz 
theorems \cite{BN,BN5}. We are currently working to remove this restriction.
\em}\end{rmk}

One application of Theorem \ref{main} is a proof of Conjecture \ref{whynot} for the completed local ring at the vertex of a cone over certain projective varieties, even though the class groups of these local rings tend to be infinitely generated. 

\begin{cor}\label{cone}
Conjecture \ref{whynot} holds for the completed local ring at the vertex $p$ of the 
cone $V$ over a smooth variety $X \subset \mathbb P^n$ satisfying 
$H^1(\O_X (k))=H^2(\O_X (k))=0$ for $k > 0$. 
\end{cor}

As in Remark \ref{condition}, we must take $n=2$ when $X$ is a curve. 
Thus examples of $X$ in Corollary \ref{cone} include plane curves 
of degree $d \leq 3$, surfaces $X \subset \mathbb P^3$ of degree $d \leq 4$ and 
Arithmetically Cohen-Macaulay subvarieties $X \subset \mathbb P^{n}$ of dimension at least $2$. 

Claborn showed in the 1960s that every abelian group is the class group of a Dedekind domain \cite{claborn}. A recent refinement of Clark shows that the Dedekind domain may be taken as the integral closure of a PID in a 
quadratic extension field \cite{clark}. 
Results like these are impossible for geometric Dedekind domains because they have trivial class group, but we are led to ask:

\begin{ques}{\em Which abstract groups are class groups of normal geometric rings? \em}\end{ques} 

Not all infinitely generated groups are possible by Remark \ref{notC}, 
but Example \ref{infinite} suggests that finite extensions of 
quotients of Abelian varieties by finite subgroups are candidates. 
For example, the quotient of the Picard group of a smooth non-rational 
plane curve $C$ by $\langle \O_C (1) \rangle$ arises in this way 
(see Proposition \ref{background} (a)). We apply Theorem \ref{main} to 
show that every finitely generated abelian group is possible: 

\begin{thm}
Let $G$ be any finitely generated abelian group. 
Then there is a point $p$ on a normal surface $S \subset \mathbb P^3$ for which 
$G \cong \Cl \O_{S,p} \subset \Cl \widehat {\O_{S,p}}$.  
\end{thm}

One of our goals here is to publicize these interesting questions. 
The main result is Theorem \ref{main} and the applications above. 
In the long run we hope that this theorem will pave the way for a 
proof of Conjecture \ref{whynot}. In section 2 we prove the theorem and 
apply it to rational double point singularities. 
Section 3 addresses vertex singularities on cones. In Section 4 we compute 
the completed class group of the singularity at $p$ of a general surface 
containing $r$ general lines passing through $p$. 

\section{Finitely generated subgroups of local class groups}

In this section we prove Theorem \ref{main} from the introduction, which 
follows immediately from Lemma \ref{intrep} and Theorem \ref{fgsubgp} below. 
We note that Conjecture \ref{3foldconj} would remove the one technical 
assumption in Theorem \ref{fgsubgp} and recover in Corollary \ref{doublepoint} 
the main theorem of \cite{BN4} with a very short proof. 

\begin{lem}\label{intrep} Let $p\in V\subset\mathbb P^n$ be a normal singularity, and let $g\in \Cl \O_{V,p}$. Then there exists an integral divisor on $V$ whose class in $\O_{V,p}$ is $g$.
\end{lem}

\begin{proof}
The element $g$ lifts to a Weil divisor $C \in \Cl V$ because height 
one primes in the local 
ring corresponds to codimension one subvarieties of $V$. 
Replacing $C$ with $C + n H$ we may assume that $C$ is 
effective and we claim that $C$ may be taken integral. 
To see this, choose $n > 0$ so that $\I_C (n)$ is generated by 
global sections. Since $\I_C$ is generically a principal ideal 
along $C$, the general section $s \in H^0(\I_C (n))$ generates 
$\I_C$ off a set of codimension two. Therefore the hypersurface $s=0$ 
on $V$ consists of $C$ along with another divisor $E$ meeting $C$ 
properly and $C + E = nH$ on $V$. 
Choosing $m > n$ so that $\I_E (m)$ is generated by global sections, 
the linear system $H^0(\I_E (m+1))$ separates points and 
tangent vectors away from the base locus $E$ and the general 
member consists of a proper union $E \cup D$. The
restricted linear system separates points and tangent vectors away 
from $E$, hence gives a map $V-{E} \to \mathbb P^N$ whose image has 
dimension $\dim V \geq 2$ and we may apply Bertini's theorem 
\cite[Thm. 6.3]{jouanolou} to see that $D$ is integral. 
Finally 
\[
D \cong mH - E \cong mH - (nH - C) \cong (m-n)H+C
\] 
so $D$ is linearly equivalent to $C$ modulo the hyperplane 
class $H$ (which is trivial in $\Cl \O_{V,p}$), hence equivalent to $g$ in $\Cl\O_{V,p}$.
\end{proof}

\begin{thm}\label{fgsubgp} 
Let $V \subset \mathbb P^n$ be a variety with normal local complete 
intersection point $p \in V$, where $n=3$ if $\dim V=2$. 
%Let $p \in V \subset \mathbb P^n$ be a normal 
%local complete intersection singularity on a variety, where $n=3$ if $\dim V=2$. 
Given prime divisors $C_i \subset V$, there is a closed subscheme 
$B \subset \mathbb P^n$ containing the $C_i$ such that the general complete 
intersection $W$ of dimension $\dim V$ containing $B$ is analytically isomorphic 
to $V$ at $p$, the isomorphism identifying $\Cl \O_{W,p} \subset \Cl \widehat \O_{W,p}$ 
with the subgroup generated by the $C_i$. 
\end{thm}

\begin{rmk}{\em
The base locus $B$ in Theorem \ref{fgsubgp} can be taken to be the union of the 
$C_i$ and a multiplicity structure supported on $\Sing V$. The existence of the 
analytic isomorphism $\psi: \widehat \O_{V,p} \to \widehat \O_{W,p}$ is given 
by our extension of Ruiz' lemma \cite[Prop. 4.1]{BN5} and the classes $C_i$ 
generate $\Cl \O_{W,p}$ by \cite[Cor. 1.5]{BN5}. 
The hard part is showing that the abstractly constructed isomorphism $\psi$ 
takes the classes $C_i$ to their expected destinations in $\Cl \widehat \O_{W,p}$.
\em}\end{rmk}

\begin{proof} 

By Lemma~\ref{intrep}, each $g_i \in G$ is represented by an integral 
divisor $C_i \subset V$ corresponding to a height one 
prime $P_i \subset \O_{V,p}$ via the ideal sheaf $\I_{C_i,V}$ 
and hence to a prime $Q_i \subset \O_{\mathbb P^n,p}$ of height $n-d+1$, 
where $d = \dim V$. 
Denote by $\widehat P_i$ and $\widehat Q_i$ the ideals generated in 
their respective completions. 

\[
\begin{array}{ccc}
Q_i & \subset & \O_{\mathbb P^n, p} \\ 
\downarrow & & \downarrow \\ 
P_i & \subset & \O_{V,p}
\end{array} 
\;\;\;\;\;\;\;\;
\begin{array}{ccc}
\widehat Q_i & \subset & \widehat \O_{\mathbb P^n, p} \\ 
\downarrow & & \downarrow \\ 
\widehat P_i & \subset & \widehat \O_{V,p}
\end{array} 
\]

Recall that a ring $A$ is {\it excellent} \cite[$\S$ 13]{M1} if it is 
Noetherian, universally catenary, a G-ring (the completion map
$A_p \to \widehat A_p$ is regular for all $p \in \Spec A$) and the regular 
locus of any finitely generated $A$-algebra $B$ is open in $\Spec B$. 
This class of rings includes complete Noetherian local rings and localizations of 
quotients of regular local rings. Thus the quotients 
$\O_{\P^n,p}/Q_i$ are excellent and hence their completions 
$\widehat{\O_{\P^n,p}/Q_i} \cong \widehat \O_{\P^n,p} / \widehat Q_i$ retain the 
property of being reduced \cite[Thm. 32.2]{M2}, therefore 
$\widehat Q_i$ can be written as an intersection of prime ideals 
$\bigcap \widehat Q_{i,j}$. Furthermore, normality of $\O_{\P^n,p}$ implies 
that of $\widehat \O_{\P^n,p}$ \cite[Thm. 32.2]{M2}, so both rings are 
equidimensional and catenary as well, since both are excellent. 
It follows \cite[Thm. 31.5]{M2} that $\widehat \O_{\P^n,p} / \widehat Q_i$ 
is equidimensional so that each $\widehat Q_{i,j}$ has height $n-d+1$ (like $Q_i$ itself). 
Likewise for the quotient rings $\O_{V,p}$ and $\widehat \O_{V,p}$ we may write 
$\widehat P_i = \bigcap \widehat P_{i,j}$ with $\widehat P_{i,j}$ (the images of the 
$\widehat Q_{i,j}$) of height one. By the uniqueness of this irredundant primary 
decomposition \cite[Thm. 18.24]{AK},  one sees that $\widehat P_i$ contains 
elements in $\widehat P_{i,j} \setminus \widehat P_{i,j}^2$,
so the valuation $v_{i,j}$ corresponding to $\widehat P_{i,j}$ satsifies $v_{i,j} (\widehat P_i) = 1$. 
Therefore in Samuel's formula \cite{samuel} for the injection 
$j_V:\Cl \O_{V,p} \to \Cl \widehat \O_{V,p}$, we have $j_V (\widehat P_i) = \sum_j \widehat P_{i,j}$, and 
$j_V (G) =  \langle \sum_j \widehat P_{1,j}, \sum_j \widehat P_{2,j}, \dots, 
\sum_j \widehat P_{r,j} \rangle$. 

Since $p \in V$ is a normal local complete intersection, $V$ is locally defined 
by an ideal $I_V = (F_1, F_2, \dots, F_{n-d})$ and the singular locus is defined by 
$I_V+J$, where $J$ is the Jacobian ideal generated by partial 
derivatives of the $F_i$, and $(F_i)+J$ has height $\geq n-d+2$ by normality. 
Now let $Z$ be the subscheme of $\mathbb P^n$ defined by the ideal 
$(\bigcap I_{C_i} \cdot J_F^2) + I_V$. Scheme-theoretically, $Z$ the union of the 
$C_i$ of dimension $d-1$ along with a multiplicity structure on the singular 
locus of dimension $\leq d-2$.

The general complete intersection $W$ of dimension $d$ containing $Z$ is locally 
defined by $F^\prime_i=F_i+T$, where $T \in J^2 I_C$. 
By \cite[Prop. 4.1]{BN5}, $W$ is analytically 
isomorphic to $V$ at $p$; 
that is, there is an automorphism $\alpha$ of $\C[[x_1, x_2, \dots x_n]]$ given 
by the change of variables $x_i\mapsto x_i+h_i, h_i\in J I_C \C[[x_1,x_2, \dots x_n]], 
1 \leq i=1 \leq n$, such that $\alpha(F_i)=F^\prime_i$, that is 
$F^\prime_i (x_1,x_2, \dots x_n) = F_i(x_1+h_1, x_2+h_2, \dots x_n+h_n)$. 
Moreover, for $f \in \widehat Q_{i,j}$, the difference  
\[
f(x_1+h_1, x_2+h_2, \dots x_n+h_n) - f(x_1, x_2, \dots x_n)
\]
clearly lies in 
$(h_1, h_2, \dots h_n) \subseteq I_C \C[[x_1,x_2, \dots x_n]] \subseteq \widehat Q_{i,j}$,
so $\alpha(\widehat Q_{i,j}) \subseteq \widehat Q_{i,j}$ and 
therefore $\alpha(\widehat Q_{i,j}) = \widehat Q_{i,j}$ since 
each $\widehat Q_{i,j}$ is prime. 
%Thus $\alpha (\widehat P_i) = \widehat P_i$ as well. 

From the preceding paragraph we see that $\alpha$ induces an isomorphism 
\[
\bar \alpha: \widehat \O_{V,p} = \C[[x_1, x_2, x_3]]/(F_i) \to \C[[x_1,x_2,x_3]]/(F^\prime_i) = 
\widehat \O_{W,p}
\] 
satisfying $\bar\alpha(\widehat P_{i,j}) = \widehat P_{W,i,j}$, 
where the subscript $W$ denotes passing to the quotient in $\O_{W,p}$ 
(or $\widehat \O_{W,p}$ in the sequel).

By our Noether-Lefschetz theorem \cite[Cor. 1.6]{BN5} (\cite[Thm. 1.1]{BN} if 
$d=2$ and $n=3$), the class group $\Cl W$ is generated by the $C_{i}$, 
and therefore the class group $\Cl \O_{W,p}$ is generated by the images of the primes 
$\widehat P_{i}$. Above we showed that these have image $\sum \widehat P_{i,j}$ in the 
completed local ring and that $\bar \alpha (\widehat P_{i,j}) = \widehat P_{W,i,j}$ and 
therefore $\bar \alpha (\sum_j \widehat P_{i,j}) = \sum_k \widehat P_{W,i,j}$. 
Putting this all together we have 
\[
\bar \alpha \circ j_V (G) =  \langle \sum_j \widehat P_{W,1,j}, \sum_j \widehat P_{W,2,j}, \dots , \sum_j \widehat P_{W,r,j} \rangle = j_W (\Cl \O_{W,p}),
\]
where $j_W$ is the natural injection $\Cl \O_{W,p} \to \Cl \widehat \O_{W,p}$, so we have constructed the local ring $\O_{W,p}$ having the desired property. 
\end{proof}

\begin{rmk}{\em Although stated for local rings at normal points on local 
complete intersection varieties, Theorem \ref{fgsubgp} has a purely 
algebraic formulation: Let $A$ be the localization of a complete intersection 
ring of finite type over $\mathbb C$ at a maximal ideal with completion 
$R = \widehat A$ and let $G \subset \Cl B$ be any finitely generated subgroup. 
Then there is a ring $B$ with $\widehat B \cong R$ such that the subgroup 
$\Cl B \hookrightarrow \Cl R$ is exactly the subgroup $G \subset \Cl A \subset \Cl R$. 
\em}\end{rmk}

We hope to remove the restriction that $n=3$ when $d=2$ of Theorem \ref{fgsubgp} 
in the future. For this it suffices to prove 
Conjecture \ref{3foldconj} below \cite[Rmk. 1.4]{BN5}. 
The conjecture holds for $X = \mathbb P^3$ \cite{BN} 
and the analogous statement holds when $\dim X \geq 4$ \cite{BN5}. 

\begin{conj}\label{3foldconj}
Let $X \subset \mathbb P^n$ be a normal threefold containing a closed subscheme 
$Z$ of dimension $\leq 1$ lying on a normal surface $S \subset X$ with curve components $Z_i$. 
Then the very general $Y \in |H^0(X,\I_Z (d))|$ for $d \gg 0$ is normal and the 
natural homomorphisms $\Cl X \to \Cl Y$ and $\bigoplus \mathbb Z \stackrel{\supp Z_i}{\to} \Cl Y$ 
induce an isomorphism $\alpha: \bigoplus \mathbb Z \oplus \Cl X \to \Cl Y$.
\end{conj}

As an application of Theorem \ref{fgsubgp} we recover our earlier result \cite{BN4} with a short proof:  

\begin{cor}\label{doublepoint}
If $(R,\fm)$ is a complete 2-dimensional rational double point singularity 
and $G \subset \Cl A$ is a subgroup, then there exists a geometric normal domain 
$(B,\fn)$ such that $\widehat B \cong R$ and $H = \Cl B \subset \Cl R$. 
\end{cor}

\begin{proof} For these singularities $\Cl R$ is a finite group, so it suffices 
to find a geometric local ring $(B, \fn)$ with $\widehat B \cong R$ and $\Cl B = \Cl R$ 
for each singularity type. 
For the $\rdA_n$ singularity with equation $xy-z^{n+1}$, 
the line $L:x=z=0$ is a generator for 
$\Cl \O_{S,p} \cong \Cl \widehat \O_{S,p} \cong \mathbb Z /(n+1) \mathbb Z$~\cite[Ex. 2.1]{picgps} (smaller subgroups of $\Cl \widehat \O_{S,p}$ are constructed in \cite[Prop. 3.3 and Prop. 3.4]{picgps}). 
Similarly \cite[Prop. 4.5]{BN4}, the $\rdD_n$ singularity is obtained by 
completing the local ring at the origin $p$ of a surface $S$ with equation $x^2+y^2 z - z^{n-1}=0$: 
for $n=2k$ the curves $x=z=0$ and $x=y-z^{k-1}=0$ generate 
$\Cl \widehat \O_{S,p} \cong \Z/2\Z \oplus \Z/2\Z$, while for $n=2k+1$ the 
curve $y=x-z^{k}=0$ generates $\Cl \widehat \O_{S,p} \cong \Z/4\Z$. 
For an $\rdE_6$ use the equation $x^2+y^3-z^4$ and the curve having ideal $(x-z^2,y)$ 
to generate $\Cl \O_{S,p} \cong \Z/3\Z$; for an $\rdE_7$, use the equation $x^2+y^3+yz^3$ 
and the line $(x,y)$ to generate $\Cl \O_{S,p} \cong \Z/2\Z$. 
A complete $\rdE_8$ is a UFD, so $\Cl \widehat \O_{S,p} = 0$.
\end{proof}

\section{Vertex singularities}

Here we study the singularity at the vertex $p$ of a cone $V$ over a 
smooth variety $X$. Our main results state that 
(a) every finitely generated abelian group occurs as the class group of a vertex singularity and (b) Conjecture \ref{whynot} holds when $X$ is a smooth arithmetically Cohen-Macaulay variety. Both follow from  
Theorem \ref{fgsubgp} and Proposition \ref{background}, which gives a 
geometric description of the groups 
$\Cl \O_{V,p} \hookrightarrow \Cl \widehat \O_{V,p}$. 
The first three parts are well-known, but the last part is new to 
our knowledge. 

\begin{prop}\label{background}
Let $X \subset \mathbb P^n$ be a smooth projective variety over an algebraically closed field with 
projective cone $V \subset \mathbb P^{n+1}$ having vertex $p$ and 
hyperplane section $H$. 
Then 
\begin{enumerate}
\item[(a)] $\Cl V \cong \Cl X$.
\item[(b)] The natural map $\Spec \O_{V,p} \to (V-H)$ induces an isomorphism on class groups 
$\displaystyle \Cl \O_{V,p} \cong \Cl (V-H) \cong \Pic X / \langle \O_X (1) \rangle.$
\item[(c)] The blow-up  $\tilde V$ of $V$ at $p$ is a $\mathbb P^1$-bundle over $X$ with exceptional divisor $E \cong X$. 
\item[(d)] The map $\iota: \Cl \O_{V,p} \to \Cl \widehat \O_{V,p}$ is a split injection. If $H^1(\O_X (k)) = H^2(\O_X (k)) = 0$ for $k \geq d$, then $\iota$ 
may be identified with the map $\Pic E_d / \langle E|_{E_d} \rangle \to \Pic E / \langle E|_E \rangle$ induced by the restriction, where $E_d$ is the $d$th infinitesimal 
neighborhood of $E \subset \tilde V$. 
\end{enumerate}
\end{prop}

\begin{proof} Parts (a) and (b) are \cite[II, Exer. 6.3]{AG} and (c) is 
\cite[V, Ex. 2.11.4]{AG}. 
For (d), let $W = \Spec \O_{V,p}$ and 
$Z = \Spec \widehat \O_{V,p}$. 
The morphisms 
$Z \to W \to (V - X)$ extend to morphisms 
$\tilde Z \to \tilde W \to (\widetilde V-X)$ of the blow-ups at $p$. 
All three are smooth and share the exceptional divisor 
$E = \Proj \oplus {\mathfrak m}_p^n / {\mathfrak m}_p^{n+1}$ because 
the quotients ${\mathfrak m}_p^n / {\mathfrak m}_p^{n+1}$ are isomorphic for 
$n \geq 1$. 
Likewise, they share the infinitesimal neighborhoods $E_k$ for $k \geq 1$ 
because the local rings $\O_p/ {\mathfrak m}_p^n$ are isomorphic for $n \geq 1$. 
Taking Picard groups gives a diagram with short exact rows in which the map 
$\alpha$ is the Mori map (\ref{inject}), geometrically realized as the 
pull-back map of line bundles:
\begin{equation}\label{five}
%\begin{array}{ccccccccc}
%\mathbb Z  & \stackrel{\cdot E}{\to} & \Pic (\tilde V - X)& \to & \Pic (\tilde V - E - X) & & \\
%\downarrow & & \downarrow \beta & & \downarrow {\overline \beta} & & \\ 
%\mathbb Z & \stackrel{\cdot E}{\to} & \Pic \tilde W & \to & \Pic (\tilde W - E) & \cong & \Cl \O_{V,p} \\ 
%\downarrow & & \downarrow \gamma & & \downarrow & & \downarrow \alpha \\ 
%\mathbb Z & \stackrel{\cdot E}{\to} & \Pic \tilde Z & \to & \Pic (\tilde Z - E) & \cong & \Cl \widehat \O_{V,p}
%\end{array}
\begin{CD}
\mathbb Z  @>{\cdot E}>> \Pic (\tilde V - X) @>>> \Pic (\tilde V - E - X) \\
@VVV @VV{\beta}V @VV{\overline \beta}V \\ 
\mathbb Z  @>{\cdot E}>>  \Pic \tilde W @>>>\Pic (\tilde W - E) \,\, \cong @. @.\,\, \Cl \O_{V,p} \\ 
@VVV @VV{\gamma}V @VVV @. \,\, @VV{\alpha}V \\ 
\mathbb Z  @>{\cdot E}>>\Pic \tilde Z @>>> \Pic (\tilde Z - E) \,\, \cong @. @.\,\, \Cl \widehat \O_{V,p}
\end{CD}
\end{equation}

The map $W \to V - X$ induces an isomorphism $\Cl (V-X) \to \Cl \O_{V,p}$ by part (b). 
Blowing up the vertex $p$ we obtain 
${\overline \beta}: \Pic (\tilde V - E - X) \cong \Pic (\tilde W - E)$, 
so $\beta$ is also an isomorphism by the Snake lemma. 
Since $\tilde V - X \to E$ is an $\mathbb A^1$-bundle, 
the pull-back $\Pic E \to \Pic (\tilde V - X)$ is an isomorphism whose 
inverse coincides with the restriction $\Pic (\tilde V - X) \to \Pic E$, 
hence the restriction $\Pic \tilde W \to \Pic E$ is also an isomorphism. 

Now consider the map $\gamma$ and let $\widehat Z$ be the formal completion of $Z$ along $E$. 
The category of coherent $\O_Z$-modules is equivalent to the category of coherent 
$\O_{\widehat Z}$-modules \cite[II, Thm. 9.7]{AG} (see also \cite[III, Thm. 5.1.4]{G}), therefore we obtain an 
isomorphism $\Pic \tilde Z \cong \Pic \widehat Z$. Furthermore, there is an isomorphism 
$\displaystyle \Pic \widehat Z \cong \lim_{\longleftarrow} \Pic E_k$
where $E_k$ is the scheme structure on $E$ given by the ideal $\I_E^k$ \cite[II, Exer. 9.6]{AG}. This allows us to express the restriction map $\Pic \tilde Z \to \Pic E$ as the composite 
\begin{equation}\label{one}
\Pic \tilde Z \cong \Pic \widehat Z \cong \lim_{\longleftarrow} \Pic E_k \to \Pic E
\end{equation} 
where the last map is the projection onto $\Pic E_1 = \Pic E$ from the inverse limit. 
The long exact cohomology sequences associated to the exact sequences 
\[
0 \to \I_E^k / \I_E^{k+1} \to \O_{E_{k+1}}^* \to \O_{E_k}^* \to 0 
\]
yield exact fragments 
\begin{equation}\label{frag}
H^1(\I_E^k / \I_E^{k+1}) \to \Pic E_{k+1} \to \Pic E_k \to H^2(\I_E^k / \I_E^{k+1}). 
\end{equation}
Making the identification $E \cong X \subset H \cong \mathbb P^{n}$, the canonical 
$\O (1)$ from the Proj construction of blow-up is both the pull-back of $\O_{\mathbb P^n} (1)$  
and given by $-E$, so $\I_E \cong \O (1)$ and $\I_E^k / \I_E^{k+1} \cong \O (k) \otimes \O_E \cong \O_X (k)$ via the isomorphism $X \cong E$. In view of (\ref{frag}) and the vanishing hypotheses we see that $\Pic E_{k+1} \to \Pic E_k$ is an isomorphism for each $k \geq d$, so the direct limit is isomorphic to $\Pic E_d$. In other words, the restriction 
map $\Pic \tilde Z \to \Pic E_d$ is an isomorphism. Restricting Daigram (\ref{five}) 
to $E \subset E_d$ gives the commuting diagram 
%\[
%\begin{array}{cccccc}
%\Pic (\tilde V - X) & \to & \Pic E_d & \to & \Pic E \\ 
%\downarrow \gamma \circ \beta & & \downarrow  & & \downarrow  \\ 
%\Pic \tilde Z & \stackrel{\cong}{\to} & \Pic E_d & \to & \Pic E
%\end{array} 
%\]

\[
\begin{CD}
0 @>>> \Pic((\tilde V - X) @>>> \Pic E_d @>>> \Pic E @>>> 0 \\
@. @VV{\gamma\circ\beta}V @VVV @VVV @. \\
0 @>>> \Pic \tilde Z @>{\cong}>> \Pic E_d @>>> \Pic E @>>> 0
\end{CD}
\]
in which the horizontal maps are restrictions. Since the composite  
$\Pic (\tilde V - X) \to \Pic E$ is an isomorphism, the vertical map 
$\Pic (\tilde V - E) \to \Pic \tilde Z$ is a splitting for the restriction 
map $\Pic E_d \to \Pic E$. When we quotient out by the subgroup 
$\langle E \rangle$ we obtain the diagram 
\[
%\begin{array}{ccccc}
%\mathbb Z & \stackrel{\cdot E}{\to} & \Pic (\tilde V - E) & \to & \Cl \O_{V,p} \\
%\downarrow & & \downarrow & & \downarrow \\ 
%\mathbb Z & \stackrel{\cdot E}{\to} & \Pic \tilde Z & \to & \Cl \widehat \O_{V,p} \\ 
%\downarrow & & \downarrow & & \downarrow \\
%\mathbb Z & \stackrel{\cdot E}{\to} & \Pic E & \to & \Cl \O_{V,p}
%\end{array}
\begin{CD}
0 @>>> \mathbb Z @>{\cdot E}>>  \Pic (\tilde V - E) @>>> \Cl \O_{V,p} @>>> 0\\
@. @VVV @VVV @VVV @.\\ 
0 @>>> \mathbb Z @>{\cdot E}>>  \Pic \tilde Z @>>> \Cl \widehat \O_{V,p} @>>> 0\\ 
@. @VVV @VVV @VVV @.\\
0 @>>> \mathbb Z @>{\cdot E}>>  \Pic E @>>> \Cl \O_{V,p} @>>> 0
\end{CD}
\]
whose rows are short exact sequences. This diagram identifies
$\Cl \O_{V,p} \cong \Pic E / \langle E|E \rangle$ and 
$\Cl \widehat \O_{V,p} \cong \Pic E_d / \langle E|_{E_d} \rangle$. 
Since the composite of the two rightmost vertical maps is the identity, 
the Mori map $\Cl \O_{V,p} \to \Cl \widehat \O_{V,p}$ 
is a split injection. 
\end{proof} 

\begin{rmk}{\em 
The Mori map (\ref{inject}) does not split in general. For example, the 
$\rdA_3$ singularity ring $R$ has completed class group $\mathbb Z / 4 \mathbb Z$ and 
there are geometric rings $A$ with completion $R$ for which $\Cl A = \mathbb Z / 2 \mathbb Z$ as a subgroup \cite{BN4}.
\em}\end{rmk}

\subsection{Finitely generated groups are class groups at vertex singularities}

\begin{thm}\label{anyfg}
Let $G$ be any finitely generated abelian group. Then there is a geometric normal 
domain $B$ of dimension two with $\Cl B \cong G$. 
\end{thm} 

\begin{proof}
Write $G \cong \Z^r \oplus \bigoplus_{i=1}^s \Z/n_i\Z$ for suitable $r,s, n_i$. 
Choose a smooth plane curve $C$ of sufficiently high degree that its genus 
satisfies $g \ge \frac{1}{2}(r+s)$. The vertex $p$ of the 
cone $S$ over $C$ has class group $\Cl \O_{S,p} \cong \Pic C/\langle\O_C(1)\rangle$ 
by Proposition \ref{background}.
Since the only degree-$0$ class in $\langle\O_C(1)\rangle$ is $0$, 
the composite map
\[
{\Pic}^0 (C) \to \Pic C \to \Pic C/\langle\O_C(1)\rangle
\]
is injective, where $\Pic^0 (C)$ is the subgroup of $\Pic C$ consisting of the 
degree-$0$ classes. 
Since $\Pic^0(C)$ is isomorphic to the Jacobian variety $J(C)$, 
which for the complex curve $C$ is isomorphic to $\C^{g}/\Lambda$ with 
$\Lambda$ a rank-$(2g)$ lattice in $\C^{g}$, we see that 
\[
{\Pic}^0 (C) \cong \R^{2g}/\Z^{2g} \cong \left(\R/\Z\right)^{2g}
\]
as an additive group. Since $\R / \Z$ has elements of all positive orders 
(including $\infty$), we can choose $r$ elements of summands having order 
$\infty$ and $s$ elements having respective orders $n_i$, which generate a 
subgroup of $\R^{2g}/ \Z^{2g}$ isomorphic to $G$. Now apply Theorem \ref{fgsubgp} 
to these elements. 
\end{proof}

\begin{rmk}\label{higher}{\em
In Theorem \ref{anyfg} above we constructed $B$ as the local ring at the 
vertex of the cone over a smooth plane curve. To obtain rings $B$ of 
higher dimension we can adjust the proof as follows. Choose $C$ as above and let 
$X = C \times \mathbb P^l$ for any $l \geq 1$. 
Then $\Pic X \cong \Pic C \oplus \mathbb Z$ \cite[II. Exer. 6.1]{AG} and 
$\O_C (1) \otimes \O_{\mathbb P^l} (1)$ is very ample on $X$, giving a 
closed immersion $X \hookrightarrow \mathbb P^N$. 
Letting $V \subset \mathbb P^{N+1}$ be the projective cone with vertex $p$, 
we obtain $\Cl \O_{V,p} \cong \Pic C \oplus \mathbb Z / \langle \O_C (1) \otimes \O_{\mathbb P^l} (1) \rangle$ and again the composite map 
\[
J(C) = {\Pic}^0 (C) \to \Pic C \oplus \mathbb Z \to \Pic C \oplus \mathbb Z / \langle \O_C (1) \otimes \O_{\mathbb P^l} (1) \rangle
\]
is injective allowing us to run the same construction. 
\em}\end{rmk}

\subsection{Cones over Arithmetically Cohen-Macaulay varieties} 

\begin{thm}\label{vanish}
Let $X \subset \mathbb P^n$ be a smooth complex variety of dimension at least two 
for which $H^1(\O_X (k))=H^2(\O_X (k))=0$ for all $k > 0$. If $V \subset \mathbb P^{n+1}$ is 
the projective cone with vertex $p$, then Conjecture \ref{whynot} holds for 
$R = \widehat \O_{V,p}$. 
\end{thm}

\begin{proof}
The vanishing hypothesis allows us to apply Proposition \ref{background} 
with $d=1$, so we obtain the isomorphism 
$\Cl \O_{V,p} = \Cl \widehat \O_{V,p}$. Now apply Theorem \ref{fgsubgp} to $V$. 
\end{proof}

\begin{cor}\label{acm} 
Let $X \subset \mathbb P^n$ be a smooth ACM subvariety of dimension $\geq 2$ 
with projective cone $V \subset \mathbb P^{n+1}$ having vertex $p$. 
Then Conjecture \ref{whynot} holds for $R = \widehat \O_{V,p}$.
\end{cor}

\begin{ex}\label{special}{\em
Corollary \ref{acm} provides many completed local rings of vertex singularities for which Conjecture \ref{whynot} holds, but to apply Theorem \ref{vanish} we only used 
the vanishings $H^1 (\O_X (k)) = H^2 (\O_X (k))=0$ for $k > 0$. 
Here are examples which need not be ACM:
\begin{enumerate}
\item[(a)] Plane curves $X \subset \mathbb P^2$ of degree $d \leq 3$. Note that 
the cone $V = C(X) \subset \mathbb P^3$, so we are staying within the restricted 
hypothesis in Theorem \ref{fgsubgp}. 
\item[(b)] Surfaces $X \subset \mathbb P^3$ of degree $d \leq 4$. 
\item[(c)] If $X \subset \mathbb P^n$ is a smooth variety 
of dimension $\geq 2$, composition with a $d$-uple embedding 
$\mathbb P^n \hookrightarrow \mathbb P^N$ for $d \gg 0$ yields 
$X \subset \mathbb P^N$ for which Theorem \ref{vanish} applies.
\end{enumerate}
\em}\end{ex}

\begin{ex}\label{mystery}{\em
The simplest case where Conjecture \ref{whynot} is unknown to us is 
the cone over a smooth plane quartic $X \subset \mathbb P^2$. 
Here all the vanishing hypotheses of Theorem \ref{vanish} hold except that 
$H^1(\O_X (1)) \cong \mathbb C$. 
Letting $V \subset \mathbb P^3$ be the cone with vertex $p$ and blow-up 
$\tilde V$ with exceptional divisor $E$ as in the proof, 
The exact sequence (\ref{frag}) takes the form 
\begin{equation}
0 \to \C \to \Pic E_2 \to \Pic E_1 \to 0
\end{equation}
for $k=1$ since $H^0(E_2, \O_{E_2}^*) \to H^0(E_1, \O_{E_1}^*)$ is surjective 
(nonzero constants are global units on $E_2$) and hence the Mori map 
$\Cl \O_{V,p} \to \Cl \widehat \O_{V,p}$ is a split injection with 
cokernel isomorphic to $\mathbb C$. We can no longer apply Theorem \ref{fgsubgp} 
with $V$, but there may well be another geometric local ring with completion 
$\widehat \O_{V,p}$ giving geometric divisors outside the image above. 
Are all the divisors in $\Cl \widehat \O_{V,p}$ geometric? 
\em}\end{ex}

\section{Surfaces containing general lines through a point}

In Theorem \ref{fgsubgp} we represented elements of a local class group
$\Cl \O_{V,p}$ by integral divisors $C_i$ with Lemma \ref{intrep} and 
carefully chose a base locus $B$ consisting of the $C_i$ and some infinitesimal data 
to pick out the group generated by them for general $W$ containing 
$B$. 
For $B = \bigcup C_i$ we do not expect the isomorphism 
$\Cl \widehat \O_{W,p} \cong \Cl \widehat \O_{V,p}$, 
but we can still expect some kind of approximation.  
In this section we give a method for comparing class groups of singularities 
having the same tangent cones and apply it to compute the local 
class group of the general surface containing a set of general lines through a 
point $p$ (Proposition \ref{rlines}) and for a general surface containing three double lines (Example \ref{3doublelines}).

Given integral divisors $C_i$ on $V \subset \mathbb P^n$ as in Theorem \ref{fgsubgp}, 
what can be said about $\Cl \O_{W,p}$ for the general complete intersection 
$W$ containing $\bigcup C_i$? If $V$ and $W$ have equal smooth tangent cones 
$E_V = E_W$ in their respective blow-ups at $p$ in $\widetilde {\mathbb P^n}$, 
we can at least make a first-order estimate in the sense that 
\[
\Cl \widehat \O_{W,p} \cong \lim_{\longleftarrow} \Pic E^n_W, \;\;\;\;\;
\Cl \widehat \O_{V,p} \cong \lim_{\longleftarrow} \Pic E_V^n
\]
where $E^n$ is the $n$th infinitesimal neighborhood of $E$. 

\begin{lem}\label{commute} Let $V$ and $W$ be subvarieties of $\P^n$ of the same dimension containing common integral divisors $C_1, C_2, \dots C_s$ 
and having an isolated singularity at $p$. Suppose further that $V$ and $W$ have equal tangent cones at $p$ that induce equal smooth exceptional divisors $E_V=E_W$ on the blowups 
$\widetilde V, \widetilde W \subset \widetilde {\mathbb P^n}$ at $p$. 
Then the diagram 
\begin{equation}\label{weakening}
%\begin{array}{ccccccc}
%\mathbb Z^s & \stackrel{C_i}{\to} &  \Cl \O_{W,p} & \subset & \Cl \widehat \O_{W,p} & \to & \Pic E_W \\ 
%|| & & & & & & || \\
%\mathbb Z^s & \stackrel{C_i}{\to} &  \Cl \O_{V,p} & \subset & \Cl \widehat \O_{V,p} & \to & \Pic E_V \\ 
%\end{array}
\begin{CD}
\mathbb Z^s @>{C_i}>>  \Cl \O_{W,p}  \,\,\subset \,\, \Cl \widehat \O_{W,p} @>>> \Pic E_W \\ 
@| \,\, @. \,\, @| \\
\mathbb Z^s @>{C_i}>>   \Cl \O_{V,p} \,\, \subset \,\, \Cl \widehat \O_{V,p} @>>> \Pic E_V 
\end{CD}
\end{equation}  
commutes, where the rightmost vertical map is the equality induced by $E_V=E_W \subset \mathbb P^{n-1}$. 
\end{lem}

\begin{proof} 
The images of the classes of $C_i$ in $\Pic E_W$ and $\Pic E_V$ are given by the 
intersection with the strict transforms 
$\widetilde C_i \subset \widetilde {\mathbb P^n}$, 
hence are equal by construction.  
\end{proof}

Now consider $r$ lines containing a fixed point $p \in \mathbb P^3$. 
When $r=2$, the general surface $S$ containing the lines is smooth 
and $\Pic S = \mathbb Z^{3}$ is freely generated by $\O (1)$ and the two 
lines \cite[Cor. 1.3]{BN}. In this section we give the answer when $r>2$ 
and the lines are in general position, 
determining the nature of the singularity on $S$ where the lines meet and which 
classes represent Cartier divisors.
It turns out that the singularity is analytically isomorphic to that of a 
cone over a plane curve. 
The following lemma tells us that the classes of are independent in the 
class group of such a cone. 

\begin{lem}\label{5pts}
Fix $r > 2$ and $d > 1$ satisfying 
\begin{equation}\label{mind}
\binom{d+1}{2} \leq r < \binom{d+2}{2}. 
\end{equation} 
Then $r$ very general points $\{p_1, p_2, \dots, p_r\} \subset
\mathbb P^2$ lie on a 
smooth curve $D$ of degree $d$ and no curve of degree $d-1$. If $r > 5$, 
then for the general such $D$ containing the points, the natural map 
$\mathbb Z^r \stackrel{p_i}{\to} \Pic D/\langle \O_D (1) \rangle$ is
injective.
\end{lem}

\begin{proof}
The first statement is clear because $r$ points $p_i \in \mathbb P^2$ in general 
position impose independent conditions on curves of fixed degree and lie on a smooth 
curve of degree $d$. For the second statement let 
$U_d \subset \mathbb P H^0 (\O_{\mathbb P^2} (d))$ be the smooth curve locus and let 
\[
W = \{(p_1,p_2, \dots p_r,D) \in (\mathbb P^2)^r \times U_d: p_i \in D \}
\]
which comes equipped with the family of curves 
\[ 
\begin{array}{ccc}
\mathcal D & \subset & \mathbb P^2 \times W \\
\pi \downarrow  &  & \downarrow \\ 
W & = & W
\end{array}
\]
obtained by pulling back the universal curve over $U_d$ along the projection $W \to U_d$. 
Let $\mathcal H$ be the pullback of $H=\O_{\mathbb P^2} (1)$ on $\mathcal D$ 
and for each $1 \leq i \leq r$ the assignment $(p_1, \dots, p_r,D) \to p_i \in D$ 
gives a section to $\pi$ whose image is an effective Cartier divisor 
${\mathcal P}_i \subset \mathcal D$. For fixed 
$(n_i,m)=(n_1,\dots,n_r,m) \in \mathbb Z^{r+1}$, let $T(n_i,m) \subset W$ defined by 
$\O_D (\sum n_i p_i + mH) = 0 \in \Pic D$. 
If $\sum n_i + md \neq 0$, then $T(n_i,m)$ is empty by reason of degree.  
If $\sum n_i + md=0$, the line bundle 
${\mathcal M} = \O_{\mathcal D} (\sum n_i {\mathcal P_i} + md {\mathcal H})$ 
restricts to $\O_D (\sum n_i p_i + mdH)$ on the fibers and has degree zero, 
hence is trivial if and only if $H^0(D,{\mathcal M}_{D}) \neq 0$, a Zariski 
closed condition by semicontinuity (${\mathcal M}$ is flat 
over $W$ by constancy of Hilbert polynomial). Thus each $T(n_i,m) \subset W$ is 
a proper closed subset of $W$, hence a very general set of $r$ points is
contained in a smooth degree $d$ curve $D$ with no non-trivial relations 
in $\Pic D/ \langle \O_D (1) \rangle$. 
\end{proof}

\begin{prop}\label{rlines} Let $Z = \bigcup_{i=1}^r L_i$ be a 
union of $r$ general lines passing through $p \in \mathbb P^3$. 
For a very general surface $S \in |H^0 (\I_Z (d)) |$ with $d \gg 0$, 
define $\theta: \mathbb Z^r \stackrel{L_i}{\to} \Cl \O_{S,p}$. 
\begin{enumerate}[(a)]
\item The surface $S$ has the same tangent cone at $p$ as the vertex 
of a cone $X$ over a smooth plane curve of minimal degree containing 
a hyperplane section of the lines. 
\item If $r \leq 2$, then $\Cl \O_{S,p} = 0$. 
\item If $3 \leq r \leq 5$, then $\Cl \O_{S,p} \cong
\mathbb Z / 2 \mathbb Z$ 
and $\theta$ takes each $L_i$ to a generator.  
\item If $r > 5$, then $\theta$ is injective. 
\end{enumerate}
\end{prop}

\begin{proof}
The surface $S$ is smooth at $p$ for $r \leq 2$, hence $\Cl \O_{S,p}=0$ 
and $\theta=0$, proving (b). 
For $r > 2$, let $Z \subset \mathbb A^3$ 
be the cone over $r$ general points in $\mathbb P^2$ with vertex $p=(0,0,0)$. 
The ideal $I_Z$ is the homogeneous ideal of the $r$ points in $\mathbb P^2$, 
so the general surface $S$ containing $Z$ has equation $F+G$, 
where $\deg F = d$ and $G$ consists of higher order terms. 
By Lemma \ref{5pts}, $F$ is the equation of the cone $X$ over a smooth 
degree $d$ plane curve $V$ for $S$ general, so $S$ and $X$ have equal tangent cones at $p$ which resolve by one blowup with equal exceptional 
curves $E_S= E_X$ isomorphic to the plane curve defined by $F$. This gives (a).

For $3\le r\le 5$, $d=2$, $X$ has an $A_1$ singularity, and thus so does $S$ \cite[Char. A$3'$]{durfee}. This, the fact that a complete $\rdA_1$ has class group $\Z/2\Z$, and the fact that the $L_i$ are non-Cartier on $S$ immediately gives (c).

For $r>5$, by Lemma~\ref{commute} we obtain the diagram
\[
\begin{CD}
\mathbb Z^r @>{L_i}>>  \Cl \O_{S,p}  \,\,\subset \,\, \Cl \widehat \O_{S,p} @>>> \Pic E_S \\ 
@| \,\, @. \,\, @| \\
\mathbb Z^r @>{L_i}>>   \Cl \O_{V,p} \,\, \subset \,\, \Cl \widehat \O_{V,p} @>>> \Pic E_V 
\end{CD}
%\begin{array}{ccccccc}
%\mathbb Z^r & \stackrel{L_i}{\to} &  \Cl \O_{S,p} & \subset & \Cl \widehat \O_{S,p} & \to & \Pic E_S \\ 
%|| & & & & & & || \\
%\mathbb Z^r & \stackrel{L_i}{\to} & \Cl \O_{X,p} & \subset & \Cl \widehat \O_{X,p} & \to & \Pic E_X \\ 
%\end{array}
\]
Since the lines are general, their classes are independent in $\Pic E_X$ 
by Lemma \ref{5pts}, so the composition along the bottom row is injective, and therefore so is that along the top, which gives (d). 
\end{proof}

%so $\Cl \O_{X,p} \cong \Pic V / \langle \O_V (1) \rangle$. Since $\deg G > 3$, 
%$G \in (x,y,z)^3$, so $\widehat \O_{S,p} \cong \widehat \O_{X,p}$ by Ruiz' lemma \cite{ruiz}. 
%\blu{Now I'm even more worried - isn't the previous sentence nonsense? $G \in (x,y,z)^3$ is 
%not enough for Ruiz' conclusion, is it? Isn't $G \in (x,y,z)^{d+1}$? But this still isn't good enough 
%for Ruiz, is it?}
%By the compatibility of Theorem \ref{fgsubgp}, \blu{I'm suddenly nervous about the application 
%of the theorem here, because we used a different base locus in the theorem. I suspect that the 
%compatibility here is OK by the PROOF of the theorem, maybe we need to say a little more here?} 
%the classes of the $r$ lines in $X$ 
%agree with the corresponding classes in $\Cl \O_{X,p}$ under the isomorphism, giving
%\[
%\mathbb Z^r \rightarrow \Pic V / \langle \O_V (1) \rangle \cong \Cl \O_{X,p} 
%\hookrightarrow \Cl \widehat \O_{X,p} \cong \Cl \widehat \O_{S,p}.
%\]
%and since the images of the lines generate $\Cl \O_{S,p}$ we also have that 
%$\Cl \O_{S,p} \cong \Cl \O_{X,p}$ under the isomorphism. 
%When $3 \leq r \leq 5$, $d=2$ and $D \cong \mathbb P^1$ is a conic with 
%$\O_D (1) \cong \O_{\mathbb P^2} (2)$ and $\Cl \O_{X,p} \cong \mathbb Z / 2 \mathbb Z$ 
%with the lines mapping to the generator. When $r > 5$ the map on the left is injective 
%by Lemma \ref{5pts}.

\begin{cor}\label{piclines}
Let $S$ and $r$ be as in Proposition \ref{rlines}. Then \\
(a) $r \leq 2 \Rightarrow \Pic S  = \Cl S$. \\
(b) $3 \leq r \leq 5 \Rightarrow \Pic S = \{\sum a_i L_i + b H \in \Cl S: 2 | \sum a_i \}$. \\
(c) $r > 5 \Rightarrow \Pic S = \{bH \in \Cl S: b \in \mathbb Z\}$.  
\end{cor}

\begin{proof}
This follows immediately from the Theorem in view of the exact sequence 
\[
0 \to \Pic S \to \Cl S \to \Cl \O_{S,p}.
\]
\cite[Prop. 2.15]{GD} and the identification 
$\Cl S \cong \mathbb Z^{r+1}$ generated by $L_i$ and $H=\O (1)$. 
\end{proof}

In Proposition \ref{rlines} it is essential that the lines be 
in general position. For example, if $C$ consists of $r$ planar 
lines passing through $p$, then the general surface $S$ 
containing $C$ is smooth at $p$ and $\Cl S = \Pic S$ is freely 
generated by $H$ and the lines $L_{i}$. A more interesting configuration is the 
pinwheel featured in the following proposition.

\begin{prop}\label{pinwheel}
For $r \geq 2$, let $C \subset \mathbb P^3$ be the union of $r$ 
general lines $L_1, L_2, \dots, L_r$ contained in a plane $H$ 
passing through $p$ and a line $L_0 \not \subset H$ through $p$. 
Then there are local coordinates $x,y,z$ and $a_i \in k$ for which 
\begin{equation}\label{pinhead}
I_C = (xz,yz,\Pi_{j=1}^{r} (x+a_j y)).
\end{equation}
The very general surface $S$ containing $C$ has an $\rdA_{r-1}$ 
singularity at $p$ and the restriction map 
$
\theta: \Cl S \to \Cl \O_{S,p} \cong \mathbb Z / r \mathbb Z
$
sends the lines $L_1, \dots L_r$ to $1$ and $L_0$ to $-1$. 
\end{prop}

\begin{proof} 
If $z=0$ is the plane containing the $r$ general lines, it is clear that the ideal 
(\ref{pinhead}) cuts out $C$ set-theoretically for some $a_j$. 
To see that this is the total ideal for $C$, observe that the closure ${\overline C}$ 
in $\mathbb P^3$ is contained in the complete intersection 
$X: \Pi (x+a_j y) = zx=0$, which links $C$ to the planar multiplicity 
$(r-1)$-line $Z$ given by $x=y^{r-1}=0$. Since $Z$ is planar, we can 
write down the total ideal for ${\overline C}$ using the cone construction 
from liaison theory. 

The general surface $S$ containing $C$ has local equation 
$0=xz - Ayz + B \Pi (x+a_j y)$ for general $A,B$. 
Setting $x^\prime = x - Ay$ 
this becomes $0=x^\prime z + B \Pi (x^\prime+(A+a_j)y)$. 
The product on the right can be written 
$x^\prime m + B \Pi (A+a_j) y^r$, so if we set $z^\prime = z+m$ 
the equation for $S$ takes the form $x^\prime z^\prime + u y^r=0$, 
where $u = B \Pi (A+a_j)$ is a unit, exhibiting the equation of an 
$\rdA_{r-1}$ singularity, which has local Picard 
group $\Cl \O_{S,p} \cong \mathbb Z / r \mathbb Z$. 

To find the images of the lines $L_j$ in $\Cl \O_{S,p}$, 
we blow up $S$ and check the intersections of the strict transforms 
$\tilde L_i$ with the exceptional divisors $E_i$, starting with $L_0$, 
for which the ideal is $(x,y)=(x^\prime,y)$. Adding projective coordinates $X,Y,Z$ 
for the blow-up and looking on the patch $Z=1$, we have 
$(x^\prime,y)=(zX,zY)=z(X,Y)$ so $\tilde L_0$ is the line $X=Y=0$, 
which meets the component $x^\prime=y=z^\prime=X=0$ of the exceptional divisor 
at the origin, a smooth point of $\tilde S$. Since this is the 
only intersection of $\tilde L_0$ and $E$, $\tilde L_0 = (1,0,...0)$ 
in $\Cl \O_{S,p}$ and is thus a generator. 

For $1 \leq j \leq r$, the ideal for $L_j$ is
\[
(z, x+a_j y)=(z,x^\prime +(A+a_j)y)=(z^\prime-m,x^\prime+(A+a_j)y)
\]
and $\tilde L_j$ meets the other component of $E$ in a single point 
away from the singularity of $\tilde S$. For example, on the patch $X=1$, 
the total transform is given by $(z^\prime-m,x^\prime + (A+a_j)y)=x(Z-m/x^\prime,1+(A+a_j)Y)$, 
so the strict transform is $(Z-m/x^\prime,1+(A+a_j)Y)$. The intersection with the exceptional 
divisor component $x=y=z=Z=0$ is given by ideal 
$(x,y,z,Z,Z-m/x^\prime,1+(A+a_j)Y)=(x,y,z,Z,1+(A+a_j)Y)$, 
which defines a single point away from the singularity of $\tilde S$. 
\end{proof}

\begin{ex}\label{3doublelines}{\em
Consider the subscheme $Z \subset \mathbb P^3$ consisting of the union of 
three double lines defined by ideals $(x,y^2), (x^2,z)$, and $(y,z^2)$ with respective supports $L_1, L_2$ and $L_3$. There is an obvious containment 
$I = (xyz,x^2 y, y^2 z, z^2 x) \subset (x,y^2) \cap (x^2,z) \cap (y,z^2)$ 
and in fact they are equal because $I$ has minimal graded 
resolution 
\[
0 \to S(-4)^3 \to S(-3)^4 \to I \to O
\]
and sheafifying gives the resolution for an ACM curve $W$ of degree $6$ 
(it's linked to a degenerate twisted cubic curve by two cubics). 
Thus $Z \subset W$ and both have degree $6$ so they are equal. 
The general surface $S$ of high degree containing $Z$ has equation of the form $Axyz + Bx^2y+ Cxz^2+Dy^2z$, where $A,B,C$, and $D$ do not vanish at 
the point $p=(0,0,0,1)$ of intersection of the lines, therefore the 
singularity of $S$ at $p$ has the same tangent cone as the vertex 
of the cone $V$ over the plane curve $C_0$ defined by 
$axyz+bx^2y+ cxz^2+dy^2z$ where $a,b,c,d$ are the respective constant 
terms, which is smooth for general $A,B,C,D$. 
The singularities resolve in one blow-up with exceptional divisors 
$E_S = E_V \cong C_0$. Thus we can apply Lemma~\ref{commute} 
to obtain a commutative diagram
\[
%\begin{array}{ccccccc}
%\mathbb Z^3 & \stackrel{L_i}{\to} &  \Cl \O_{S,p} & \subset & \Cl \widehat \O_{S,p} & \to & \Pic E_S \\ 
%|| & & & & & & || \\
%\mathbb Z^3 & \stackrel{L_i}{\to} &  \Cl \O_{V,p} & \subset & \Cl \widehat \O_{V,p} & \to & \Pic E_V \\ 
%\end{array}
\begin{CD}
\mathbb Z^3 @>{L_i}>> \Cl \O_{S,p} \,\, \subset \,\, \Cl \widehat \O_{S,p} @>>> \Pic E_S \\ 
@| \,\, @. \,\, @| \\
\mathbb Z^3 @>{L_i}>>   \Cl \O_{V,p} \,\, \subset \,\, \Cl \widehat \O_{V,p} @>>> \Pic E_V 
\end{CD}
\]

%To see this, look on a patch, {\em e.g.} $Z=1$ in the blow-up; local coordinates on this patch have the exceptional curve defined by the ideal $(z, XY+aX^2Y+bX+cY^2)$. \blu{Question: should we adjust the base locus as in the proof of the main theorem to force this one 
%to be analytically isomorphic to the vertex singularity? I ask because otherwise I'm not sure why 
%$\Cl \O_{V,p}$ is isomorphic to $\Pic C_0/\langle\O_{C_0}(1)\rangle$ as we state below, though maybe 
%you have another reason for this.} \ed{I spent some time today looking at some literature; there are some related examples worked out in this paper \href{url}{http://www.imath.kiev.ua/~drozd/surfsing.pdf} where the self-intersection of $E$ is more or less blithely stated to be there right thing. But I don't know; it is surely an easy fix to soup up the singularity at the origin to force a cone singularity, so maybe that's what we should do anyway.}  
By the main theorem of~\cite{BN}, the class group $\Cl S$ is freely generated by the lines $L_i$ and $\O_S (1)$; The images under 
the surjective map 
$\Cl S \to \Cl \O_{V,p} \cong \Pic C_0/\langle\O_{C_0}(1)\rangle$ correspond via the construction of blowing up and completing along 
the exceptional curve to the classes in corresponding to the 
points $P_1=(0,0,1), P_2=(0,1,0)$, and $P_3=(1,0,0)$ respectively. 
The line $x=0$ intersects $C_0$ in a double point at $P_1$ and a 
reduced point at $P_2$, hence $2P_1+P_2\in\O_{C_0}(1)$ and thus 
$P_2=-2P_1$ in $\Cl \O_{S,p}$. 
Likewise $P_3=-2P_2=4P_1$ and $P_1=-2P_3=-8P_1$, so 
$\Cl \O_{S,p} \cong \Z/9\Z$. 
\em}\end{ex}

\end{document}